\documentclass[12pt,a4paper]{article}
\usepackage{amsmath,amsthm,amssymb, mathrsfs}
\usepackage{hyperref}
\usepackage{cite}
\input xy
\xyoption{all}
\usepackage{epsfig}
\newtheorem{thm}{Theorem}[section]

\newtheorem{lem}[thm]{Lemma}

 \theoremstyle{definition}

\theoremstyle{remark}
\newtheorem{remark}[thm]{Remark}

\numberwithin{equation}{section}
\newcommand{\ben}{\begin{enumerate}}

\newcommand{\een}{\end{enumerate}}
\newcommand{\bit}{\begin{itemize}}
\newcommand{\eit}{\end{itemize}}

\begin{document}

\title
{Revisiting the Problem of Recovering Functions in $\Bbb R^{n}$ by Integration on $k$ Dimensional Planes}
\date{}

\vskip 1cm
\author{Yehonatan Salman \\ Email: salman.yehonatan@gmail.com\\ Weizmann Institute of Science}
\date{}

\maketitle

\begin{abstract}

The aim of this paper is to present inversion methods for the classical Radon transform which is defined on a family of $k$ dimensional planes in $\Bbb R^{n}$ where $1\leq k\leq n - 2$. For these values of $k$ the dimension of the set $\mathcal{H}(n,k)$, of all $k$ dimensional planes in $\Bbb R^{n}$, is greater than $n$ and thus in order to obtain a well-posed problem one should choose proper subsets of $\mathcal{H}(n,k)$. We present inversion methods for some prescribed subsets of $\mathcal{H}(n,k)$ which are of dimension $n$.

\end{abstract}

\section{Introduction and Motivation}

\subsection{Basic mathematical background}

\hskip0.6cm For an integer $n\geq2$ denote by $\Bbb R^{n}$ the $n$ dimensional Euclidean space. For $p\in\Bbb R^{n}$ and $r\geq0$ denote by
\vskip-0.2cm
$$\hskip-8.15cm \mathrm{S}(p, r) = \{x\in\Bbb R^{n}:|x - p| = r\}$$
the hypersphere in $\Bbb R^{n}$ with the center at $p$ and radius $r$. Let $\Bbb S^{n - 1} = \mathrm{S}(\overline{0}, 1)$ be the unit sphere in $\Bbb R^{n}$ and for every $\psi\in\Bbb S^{n - 1}$ denote by $\Bbb S_{\psi}^{n - 2}$ the great $n - 2$ dimensional subsphere of $\Bbb S^{n - 1}$ which is orthogonal to $\psi$:
\vskip-0.2cm
$$\hskip-8.3cm \Bbb S_{\psi}^{n - 2} = \{x\in\Bbb S^{n - 1}:\langle x, \psi\rangle = 0\}$$
where $\langle, \rangle$ denotes the usual inner product in $\Bbb R^{n}$. For an integer $k$ satisfying $1\leq k\leq n - 1$ denote by $\mathcal{H}(n,k)$ and by $\mathcal{H}^{\ast}(n,k)$ respectively the families of all $k$ dimensional planes ($k$ planes for short) and $k$ dimensional subspaces (i.e., $k$ planes passing through the origin) in $\Bbb R^{n}$.

For a continuous function $f$, defined in $\Bbb R^{n}$, define the $k$ dimensional Radon transform $\mathcal{R}_{k}f$ of $f$ by
\vskip-0.2cm
$$\hskip-10.4cm \mathcal{R}_{k}f:\mathcal{H}(n,k)\rightarrow\Bbb R,$$
$$\hskip-7.9cm (\mathcal{R}_{k}f)(\Sigma) = \int_{\Sigma}fdm_{\Sigma}, \Sigma\in\mathcal{H}(n,k)$$
where $dm_{\Sigma}$ denotes the standard infinitesimal area measure on the $k$ plane $\Sigma$.

\subsection{The main problem and known results}

\hskip0.6cm The aim of this paper is to find inversion methods for the $k$ dimensional Radon transform $\mathcal{R}_{k}$ in case where $n\geq3$ and $1\leq k\leq n - 2$. That is, we would like to express a function $f$ in question via its $k$ dimensional Radon transform $\mathcal{R}_{k}f$.

In case where $k = n - 1$ the $k$ dimensional Radon transform $\mathcal{R}_{k} = \mathcal{R}_{n - 1}$ coincides with the classical Radon transform $\mathcal{R}$ which integrates functions on the set of all hyperplanes in $\Bbb R^{n}$. There is an extensive literature regarding the classical Radon transform where the first results can be tracked down to 1917 when J. Radon showed in \cite{12} that any differentiable function $f$ in $\Bbb R^{3}$ can be recovered from its Radon transform $\mathcal{R}f$ (i.e., where our data consists of the integrals of $f$ on hyperplanes in $\Bbb R^{3}$). Further results concerning the problem of recovering a function $f$ from $\mathcal{R}f$ in higher dimensions as well as finding uniqueness and range conditions for $\mathcal{R}$ have been obtained by many authors and we do not intend to give an exhaustive list of references. Some classical results can be found in \cite{2, 4, 5, 6, 7, 8, 9, 10, 11, 12, 13, 14, 15}.

For the case where $1\leq k\leq n - 2$ and where $\mathcal{R}_{k}$ is defined on the whole family $\mathcal{H}(n,k)$ of $k$ planes in $\Bbb R^{n}$ then obviously the reconstruction problem becomes trivial since one can use known inversion methods for the case where integration is taken on hyperplanes in $\Bbb R^{n}$ and the fact that each such hyperplane is a disjoint union of $k$ planes. However, if $1\leq k\leq n - 2$ then $\dim\mathcal{H}(n,k) = (k + 1)(n - k) > n$ and the inversion problem becomes over determined. Thus, in order to obtain a well posed problem one should restrict the domain of $\mathcal{R}_{k}f$ to a subset $\mathcal{H}'\subset\mathcal{H}(n,k)$ for which $\dim\mathcal{H}' = n$.

An extensive work and results concerning the problem of determining a function $f$ from $\mathcal{R}_{k}f$ in case where the later transform is restricted to various subsets $\mathcal{H}'\subset\mathcal{H}(n,k)$ of dimension $n$ have been obtained in \cite{10, 11}. In \cite[Chap. 4.2]{11} a reconstruction formula was found in case where $\mathcal{R}_{1}f$ is restricted to a subset $\mathcal{H}'\subset\mathcal{H}(n,1)$ of lines which meet a curve $\mathrm{C}\subset\textbf{P}(\Bbb R^{n})$ (where $\textbf{P}(\Bbb R^{n})$ is the projective space of all lines in $\Bbb R^{n}$ which pass through the origin) at infinity. The dimension of this subset coincides with the dimension of $n$ in case where $n = 3$.

In \cite[Chap. 2.6]{10} and \cite[Chap. 4.3]{11} a reconstruction formula was obtained for $\mathcal{R}_{1}f$ in case where $\mathcal{H}'\subset\mathcal{H}(3, 1)$ is the set of all lines which intersect a curve $\Gamma\subset\Bbb R^{3}, \Gamma\cap\textrm{supp}(f) = \emptyset$ which satisfies the completeness condition: each hyperplane which passes through $\textrm{supp}(f)$ intersects $\Gamma$ transversely at a unique point. In \cite[Chap. 4.4]{11} the more general case, where lines are tangent to a given surface in $\Bbb R^{3}$, was considered (every curve $\Gamma = \Gamma (s)$ can be approximated by the family of surfaces $\mathrm{\Lambda}_{\epsilon}(s, v) = \Gamma(s) + \epsilon v, v\in \Bbb S^{1}_{\Gamma'(s)/ |\Gamma'(s)|}$ as $\epsilon\rightarrow0^{+}$). For cases where integration is taken on planes of dimension $k > 1$ it is shown in \cite[Chap. 5.1]{11} how to obtain an inversion formula in case where $\mathcal{H}'\subset\mathcal{H}(n,k)$ consists of $k$ planes of the form $x + L$ where $L\in \mathbf{C}, x\in L^{\top}$ and $\mathbf{C}$ is a $k$ dimensional manifold in $\mathcal{H}^{\ast}(n,k)$.\\

\subsection{Main results}

\hskip0.6cm In this paper we present three inversion methods for three different prescribed subsets $\mathcal{H}'$ of $\mathcal{H}(n,k)$ where for each such subset the problem of recovering a function $f$, defined in $\Bbb R^{n}$, from its $k$ dimensional Radon transform $\mathcal{R}_{k}f$ is well posed (i.e., the dimension of the manifold $\mathcal{H}'$ of $k$ planes is equal to $n$).

The first inversion method is obtained for a special subset $\mathcal{H}'\subset\mathcal{H}(n,k)$ which is determined by fixing $k + 1$ points $x_{1},...,x_{k + 1}$ in $\Bbb R^{n}$ in general position. Each $k$ plane in $\mathcal{H}'$ is determined by taking a point $p$ in $\Bbb R^{n}$, which is in general position with the other fixed points, and then taking the intersection of the $k + 1$ plane passing through $p, x_{1},..., x_{k + 1}$ and the unique hyperplane which passes through $p$ and its normal is in the direction $\overrightarrow{op}$. We prove that our inversion problem is well posed and provide an inversion method for each function, whose $k$ dimensional Radon transform $\mathcal{R}_{k}f$ is defined on this subset $\mathcal{H}'$, and which satisfies a decaying condition at infinity. The proof of the reconstruction method is purely geometrical and does not use any analytical tools. We find a family of $k + 1$ planes such that the union of these planes is dense in $\Bbb R^{n}$ and such that for every $k + 1$ plane $\Sigma$ in this family the set of $k$ planes in $\mathcal{H}'$, which are contained in $\Sigma$, is in fact the set of all $k$ planes in $\Sigma$. Thus, we can use known inversion formulas for the classical Radon transform, on each such $k + 1$ plane, in order to reconstruct the function $f$ in question.

The second inversion method is obtained for the subset $\mathcal{H}'\subset\mathcal{H}(3,1)$ of all lines in $\Bbb R^{3}$ which are at equal distance from a fixed given point $p$. We show that our problem is well posed (i.e., in this case $\dim\mathcal{H}' = 3$) and obtain an inversion method using expansion into spherical harmonics and exploiting the fact that $\mathcal{H}'$ is invariant under rotations with respect to the point $p$. Observe that $\mathcal{H}'$ can be also described as the set of lines which are tangent to a given sphere with center at $p$.

The third inversion method is obtained for the subset $\mathcal{H}'\subset\mathcal{H}(3,1)$ of all lines in $\Bbb R^{3}$ which have equal distances from two given fixed points $x$ and $y$ in $\Bbb R^{3}$ ($x\neq y$). Again, we show that our problem is well posed and for the inversion method we exploit the fact that $\mathcal{H}'$ is invariant under rotations with respect to the line $l$ passing through the points $x$ and $y$. Using this invariance property of $\mathcal{H}'$ we take any hyperplane $\mathbf{H}$, which is orthogonal to $l$, and expand the restriction $f|_{\mathbf{H}}$, of the function $f$ in question, into Fourier series with respect to the angular variable in $\mathbf{H}$.\\

We would like again to emphasize that the problem of recovering functions from integration on $k$ planes, which belong to an $n$ dimensional manifold $\mathcal{H}'\subset\mathcal{H}(n,k)$, is not new and the manifolds of lines, corresponding to the second and third inversion problems described above, were already considered in \cite{10, 11}. However, the inversion methods presented in this paper are new and can provide new insights on developing other reconstruction methods for more complex subsets $\mathcal{H}'$ of $k$ planes in $\mathcal{H}(n,k)$.

\section{Exact Formulations and Proofs of the Main Results}

\hskip0.6cm In this section we formulate and prove the inversion methods for the problem of recovering a function $f$ from integration on any subset of $k$ planes from the three subsets $\mathcal{H}'\subset\mathcal{H}(n,k)$ which were described in the previous section. For the first subset $\mathcal{H}'$ of $k$ planes we do not give an explicit inversion formula but instead we just assert that every function can be reconstructed from integration on $k$ planes from $\mathcal{H}'$ whereas the inversion method is described in the proof itself. This is because we rely on the reduction to the classical Radon transform from which one can choose many known inversion formulas for this integral transform. The reconstruction methods for the other two subsets are given explicitly.

The description of the subsets $\mathcal{H}'$ of $k$ planes and the inversion methods, for the three subsets described above, are given in the following three subsections. Each subsection is self contained and the mathematical tools, the main results and their proofs for the reconstruction formulas, for each one of these three subsets, are formulated individually in each subsection.

\subsection{The case of integration on $k$ planes determined by choosing $k + 1$ points in $\Bbb R^{n}$}
\hskip0.6cm For $k + 1$ arbitrary points $x_{1},...,x_{k + 1}$ in $\Bbb R^{n}$ which are assumed to be in general position, i.e., which satisfy the condition
\vskip-0.2cm
\begin{equation}\hskip-8.75cm \textbf{rank}\left(\begin{array}{c}
x_{k + 1} - x_{1}\\
x_{k + 1} - x_{2}\\
...\\
x_{k + 1} - x_{k}\end{array}\right) = k,\end{equation}
denote by $\mathbf{H}\{x_{1},...,x_{k + 1}\}$ the unique $k$ plane passing through the points $x_{1},...,x_{k + 1}$. For every $x_{0}\in\Bbb R^{n}\setminus\{0\}$ denote by $\mathbf{H}_{x_{0}}$ the unique hyperplane whose closest point to the origin is $x_{0}$, i.e.,
\vskip-0.2cm
$$\hskip-7.4cm \mathbf{H}_{x_{0}} = \{x\in\Bbb R^{n}: \langle x - x_{0}, x_{0}\rangle = 0\}.$$
For every $m$ points $x_{1},...,x_{m}$ in $\Bbb R^{n}$ define the following set
$$\hskip-1.45cm\mathbf{U} = \mathbf{U}\{x_{1},...,x_{m}\} = \{x\in\Bbb R^{n}:\langle x - x_{1}, x\rangle = ... = \langle x - x_{m}, x\rangle = 0\}.$$
Observe that in particular $0\in\mathbf{U}$.

We choose the subset $\mathcal{H}'$ in the following way. We fix $k + 1$ points $x_{1},...,x_{k + 1}$ in $\Bbb R^{n}$ which are assumed to be in general position and then, for each $x_{0}\in\Bbb R^{n}\setminus\left(\mathbf{H}\{x_{1},...,x_{k + 1}\}\cup\mathbf{U}\{x_{1},...,x_{k + 1}\}\right)$, we define $\Sigma_{x_{0}}$ to be the unique $k$ plane which is the intersection of the $n - 1$ dimensional hyperplane $\mathbf{H}_{x_{0}}$ with the $k + 1$ plane $\mathbf{H}\{x_{0}, x_{1},..., x_{k + 1}\}$. More explicitly, we define
\newpage
$$\hskip-4cm\mathcal{H'} = \mathcal{H'}_{x_{1},...,x_{k + 1}} = \{\Sigma_{x_{0}} = \mathbf{H}\{x_{0},x_{1},...,x_{k + 1}\}\cap \mathbf{H}_{x_{0}}$$
$$\hskip2.5cm:x_{0}\in\Bbb R^{n}\setminus\left(\mathbf{H}\{x_{1},...,x_{k + 1}\}\cup\mathbf{U}\{x_{1},...,x_{k + 1}\}\right)\}.$$
Observe that for every plane $\Sigma_{x_{0}}$ we have that $x_{0}\in\Sigma_{x_{0}}$.

Now we need to show that the dimension of the set $\mathcal{H'} = \mathcal{H'}_{x_{1},...,x_{k + 1}}$ is equal to $n$ and that each plane in $\mathcal{H'}$ is of dimension $k$.

To prove that $\dim\mathcal{H'} = n$ observe that since $\mathcal{H'}$ is parameterized via the set $\Bbb R^{n}\setminus\left(\mathbf{H}\{x_{1},...,x_{k + 1}\}\cup\mathbf{U}\{x_{1},...,x_{k + 1}\}\right)$ which is of dimension $n$ we have that $\dim\mathcal{H'}\leq n$. In order to prove that $\dim\mathcal{H'} = n$ we will prove that for any two different points $x'$ and $x''$ the planes $\Sigma_{x'}$ and $\Sigma_{x''}$ are different. Indeed, since $\Sigma_{x'}\subset \mathbf{H}_{x'}$ it follows that its closest point to the origin is $x'$ and in the same way $x''$ is the closest point of $\Sigma_{x''}$ to the origin. Hence, if $\Sigma_{x'} = \Sigma_{x''}$ it follows that both $x'$ and $x''$ are the closest points of a $k$ plane to the origin and they are both different. This obviously leads to a contradiction.

To show that for each $x_{0}\in\Bbb R^{n}$ the plane $\Sigma_{x_{0}}$ is of dimension $k$ we first observe that since $x_{0}\in\Sigma_{x_{0}}$ then $\Sigma_{x_{0}}$ is a nonempty intersection of a hyperplane and a $k + 1$ plane in $\Bbb R^{n}$. Hence, $\dim\Sigma_{x_{0}}$ is equal to $k$ or $k + 1$ where the later case occurs when $\mathbf{H}\{x_{0},...,x_{k + 1}\}$ is contained in $\mathbf{H}_{x_{0}}$. This will imply in particular that $x_{1},...,x_{k + 1}\in\mathbf{H}_{x_{0}}$ and thus
\vskip-0.2cm
$$\hskip-6.3cm\langle x_{0} - x_{1}, x_{0}\rangle = ... = \langle x_{0} - x_{k + 1}, x_{0}\rangle = 0$$
which is a contradiction to the assumption that $x_{0}\notin\mathbf{U}\{x_{1},...,x_{k + 1}\}$.

Now, for every function $f$, defined in $\Bbb R^{n}$, which decays to zero fast enough at infinity, our aim is to recover $f$ from $\mathcal{R}_{k}f$ which is now restricted to the smaller set $\mathcal{H'}$ and thus our problem is well posed. For this we have the following result.

\begin{thm}

Let $f$ be a function in $C^{\infty}(\Bbb R^{n})$ (the space of infinitely differentiable functions in $\Bbb R^{n}$) which satisfies the decaying condition $f(x) = o(|x|^{-N})$ for some $N > k + 1$. Then, $f$ can be recovered from its $k$ dimensional Radon transform $\mathcal{R}_{k}f$ restricted to the set $\mathcal{H'}$ of $k$ dimensional planes.

\end{thm}

\begin{proof}

From here and after we will assume, without loss of generality, that $\{x_{k + 1} - x_{1},...,x_{k + 1} - x_{k}\}$ is an orthonormal set. Since if this is not the case then by equation (2.1) it follows that there exists an orthonormal set of vectors $\{y_{1},...,y_{k}\}$ such that
\vskip-0.2cm
$$\hskip-4cm\textbf{\textrm{span}}\{y_{1},...,y_{k}\} = \textbf{\textrm{span}}\{x_{k + 1} - x_{1},...,x_{k + 1} - x_{k}\}.$$
Hence, if we denote $x_{i}^{\ast} = x_{k + 1} - y_{i}$ then it is easily verified that
$$\hskip-5.5cm\mathbf{H}\{x_{1},...,x_{k},x_{k + 1}\} = \mathbf{H}\{x_{1}^{\ast},...,x_{k}^{\ast}, x_{k + 1}\}.$$
Hence, for the set $\{x_{1}^{\ast},...,x_{k}^{\ast}, x_{k + 1}\}$ we obtain the same set of $k$ planes, i.e.,
$$\hskip-7.75cm\mathcal{H}_{x_{1},...,x_{k}, x_{k + 1}} = \mathcal{H}_{x_{1}^{\ast},...,x_{k}^{\ast}, x_{k + 1}}$$
and we also have that $\{x_{k + 1} - x_{1}^{\ast},...,x_{k + 1} - x_{k}^{\ast}\}$ is an orthonormal set.\\

The proof of Theorem 2.1 is divided into 5 parts.

In the first part we find necessary and sufficient conditions on a point $x_{0}$ in $\Bbb R^{n}\setminus\left(\mathbf{H}\{x_{1},...,x_{k + 1}\}\cup\mathbf{U}\{x_{1},...,x_{k + 1}\}\right)$ so that the $k$ plane $\Sigma_{x_{0}}$ is contained in a given $k + 1$ plane $\Sigma'$.

In the second part we define a special subset $\mathcal{H}'(n, k + 1)\subset\mathcal{H}(n, k + 1)$ of $k + 1$ planes.

In the third part we find, with the help of the first part, necessary and sufficient conditions on a point $x_{0}$ in $\Bbb R^{n}\setminus\left(\mathbf{H}\{x_{1},...,x_{k + 1}\}\cup\mathbf{U}\{x_{1},...,x_{k + 1}\}\right)$ for the $k$ plane $\Sigma_{x_{0}}$ to be contained in a given $k + 1$ plane in $\mathcal{H}'(n, k + 1)$.

In the fourth part we show that the union of all the $k + 1$ planes in $\mathcal{H}'(n, k + 1)$ is the whole space $\Bbb R^{n}$.

Lastly, in the fifth part we show that the $k$ planes in $\mathcal{H}'$, which are contained in a given $k + 1$ plane $\Sigma'$ in $\mathcal{H}'(n, k + 1)$, are in fact all the $k$ planes in $\Sigma'$ and thus the function $f$, in Theorem 2.1, can be recovered on $\Sigma'$. Since, by the fourth part, the union of the $k + 1$ planes in $\mathcal{H}'(n, k + 1)$ is the whole space $\Bbb R^{n}$ it follows that $f$ can be recovered in $\Bbb R^{n}$. This will finish the proof of Theorem 2.1.\\

\textbf{$\bullet$ Finding conditions on the point $x_{0}$ so that the $k$ plane $\Sigma_{x_{0}}$ is contained in a given $k + 1$ plane $\Sigma'$}: For $x_{0}\in\Bbb R^{n}\setminus\left(\mathbf{H}\{x_{1},...,x_{k + 1}\}\cup\mathbf{U}\{x_{1},...,x_{k + 1}\}\right)$ observe that the $k$ plane $\Sigma_{x_{0}}$ in $\mathcal{H'}$ can be parameterized as follows
\vskip-0.2cm
$$\hskip-5cm\Sigma_{x_{0}} = \left\{x_{0} + \lambda_{1}(x_{1} - x_{0}) + ... + \lambda_{k + 1}(x_{k + 1} - x_{0})\right.$$
$$\left.:\lambda_{1}\langle x_{1} - x_{0}, x_{0}\rangle + ... + \lambda_{k + 1}\langle x_{k + 1} - x_{0}, x_{0}\rangle = 0, \lambda\in\Bbb R^{k + 1}\right\}.$$
From the above parametrization we obtained for $\Sigma_{x_{0}}$ it follows that for every $\omega\in\Bbb S^{n - 1}$ and $c\in\Bbb R$ the $k$ plane $\Sigma_{x_{0}}$ is contained in the hyperplane $\langle x, \omega \rangle = c$ if and only if the following two conditions are satisfied
\vskip0.2cm
(i) $\langle x_{0}, \omega\rangle = c$,
\vskip0.2cm
(ii) There exists an index $1\leq i\leq k + 1$ such that $\langle x_{i} - x_{0}, x_{0}\rangle\neq0$ and
\vskip-0.8cm
$$\hskip0.65cm\langle x_{i} - x_{0}, x_{0}\rangle\langle x_{j} - x_{0}, \omega\rangle = \langle x_{j} - x_{0}, x_{0}\rangle\langle x_{i} - x_{0},\omega\rangle\hskip0.25cm\textrm{for}\hskip0.25cm j = 1,...,i - 1,i + 1,...,k + 1.$$
\vskip0.2cm
Now, for a given $k + 1$ plane $\Sigma'$ we would like to find necessary and sufficient conditions on $x_{0}$ so that the $k$ plane $\Sigma_{x_{0}}$ is contained in $\Sigma'$. Let us assume that $\Sigma'$ is given by the following system of equations
\vskip-0.2cm
\begin{equation}\hskip-8.5cm\langle x, \omega_{i}\rangle = \mu_{i}, i = 1,...,n - k - 1\end{equation}
where $\omega_{1},...,\omega_{n - k - 1}$ form an orthonormal set and where $\mu_{1},...,\mu_{n - k - 1}\in\Bbb R$. From the above analysis it follows that the $k$ plane $\Sigma_{x_{0}}$ is contained in a $k + 1$ plane $\Sigma'$, given by the intersection of all hyperplanes given by (2.2), if and only if the following two conditions are satisfied
\vskip0.2cm
(i*) $\langle x_{0}, \omega_{i}\rangle = \mu_{i}, i = 1,...,n - k - 1$,
\vskip0.2cm
(ii*) For every $1\leq i'\leq n - k - 1$ there exists an index $1\leq i\leq k + 1$ such that

$\langle x_{i} - x_{0}, x_{0}\rangle\neq0$ and such that
\vskip-0.75cm
$$\hskip0.65cm\langle x_{i} - x_{0}, x_{0}\rangle\langle x_{j} - x_{0}, \omega_{i'}\rangle = \langle x_{j} - x_{0}, x_{0}\rangle\langle x_{i} - x_{0}, \omega_{i'}\rangle\hskip0.25cm\textrm{for}\hskip0.25cm j = 1,...,i - 1,i + 1,...,k + 1.$$

\textbf{$\bullet$ Defining the set $\mathcal{H}'(n, k + 1)$ of $k + 1$ planes}: Now, let us look only on $k + 1$ planes given as the intersection of the $n - k - 1$ hyperplanes given by (2.2) where $\omega_{1},...,\omega_{n - k - 1}$ satisfy
\vskip-0.2cm
\begin{equation}\hskip-5cm\langle x_{k + 1} - x_{i}, \omega_{j}\rangle = 0, i = 1,...,k, j = 1,..., n - k - 1\end{equation}
and where the parameters $\mu_{1},...,\mu_{n - k - 1}$ are then given by the following equations
\vskip-0.2cm
\begin{equation}\hskip-3.25cm\mu_{j} = \langle x_{k + 1}, \omega_{j}\rangle ( = \langle x_{i}, \omega_{j}\rangle, i = 1,...,k), j = 1,..., n - k - 1.\end{equation}
Denote this family of $k + 1$ planes by $\mathcal{H}'(n, k + 1)$. That is,
\vskip-0.2cm
$$\hskip-11.65cm\mathcal{H}'(n, k + 1)$$
$$\hskip-3.15cm = \{\Sigma_{\omega}' = \{x\in\Bbb R^{n}:\langle x - x_{k + 1}, \omega_{i}\rangle = 0, i = 1,...,n - k - 1\}$$
$$\hskip0.6cm:\omega = (\omega_{1},...,\omega_{n - k - 1}), \mathrm{where\hskip0.1cm} \omega_{i}, i = 1,...,n - k - 1 \mathrm{ \hskip0.1cm form \hskip0.1cm an \hskip0.1cm orthonormal} $$
$$\hskip-1.5cm\mathrm{set \hskip0.1cm in \hskip0.1cm} \Bbb S^{n - 1} \mathrm{\hskip0.1cm and \hskip0.1cm belong \hskip0.1cm  to \hskip0.1cm } \textbf{\textrm{span}}\{x_{k + 1} - x_{1},...,x_{k + 1} - x_{k}\}^{\bot}\}.$$

\textbf{$\bullet$ Finding conditions on the point $x_{0}$ so that the $k$ plane $\Sigma_{x_{0}}$ is contained in the $k + 1$ plane $\Sigma_{\omega}'$ in $\mathcal{H}'(n, k + 1)$}: Let us take a $k + 1$ plane $\Sigma_{\omega}'$ in $\mathcal{H}'(n, k + 1)$. Then, for a $k$ plane $\Sigma_{x_{0}}$ to be contained in $\Sigma_{\omega}'$ the point $x_{0}$ in $\Bbb R^{n}\setminus\left(\mathbf{H}\{x_{1},...,x_{k + 1}\}\cup\mathbf{U}\{x_{1},...,x_{k + 1}\}\right)$ needs to satisfy only condition (i*) since then condition (ii*) is also satisfied. Indeed, since $x_{0}\notin\mathbf{U}\{x_{1},...,x_{k + 1}\}$ there exists an index $i$ such that $\langle x_{i} - x_{0}, x_{0}\rangle\neq0$. Using the fact that for $\Sigma_{\omega}'$ we have, from equation (2.3), that $\langle x_{i}, \omega_{i'}\rangle = \langle x_{k + 1}, \omega_{i'}\rangle$, it follows that
\vskip-0.2cm
$$\hskip-3.3cm\langle x_{i} - x_{0}, \omega_{i'}\rangle = \langle x_{i}, \omega_{i'}\rangle  - \langle x_{0}, \omega_{i'}\rangle = \langle x_{k + 1}, \omega_{i'}\rangle  - \mu_{i'} = 0$$
where in the second passage we used condition (i*) and in the last passage we used equation (2.4). In the same way we can show that $\langle x_{j} - x_{0}, \omega_{i'}\rangle = 0$ and thus condition (ii*) is obviously satisfied.

Thus, it follows that the $k$ plane $\Sigma_{x_{0}}$ is contained in the $k + 1$ plane $\Sigma_{\omega}'$ in $\mathcal{H}'(n, k + 1)$ if and only if
\vskip-0.2cm
\begin{equation}\hskip-6.5cm\langle x_{0}, \omega_{i}\rangle = \langle x_{k + 1}, \omega_{i}\rangle, i = 1,...,n - k - 1.\end{equation}

\textbf{$\bullet$ The $k + 1$ planes in $\mathcal{H}'(n,k + 1)$ cover the whole space $\Bbb R^{n}$}: Our aim is to reconstruct $f$ on every $k + 1$ plane $\Sigma_{\omega}'$ in $\mathcal{H}'(n, k + 1)$. If this can be done then $f$ can be recovered on the whole of $\Bbb R^{n}$. Indeed, we only need to show that for every $z_{0}\in\Bbb R^{n}$ there exists a $k + 1$ plane $\Sigma_{\omega}'$ in $\mathcal{H}'(n, k + 1)$ such that $z_{0}\in\Sigma_{\omega}'$. That is, we need to find $n - k - 1$ orthonormal vectors $\omega_{1},...,\omega_{n - k - 1}$ such that each of them is orthogonal to $x_{k + 1} - x_{1},...,x_{k + 1} - x_{k}$ and such that the following equations are satisfied
\vskip-0.2cm
$$\hskip-6.5cm\langle z_{0}, \omega_{i}\rangle = \langle x_{k + 1}, \omega_{i}\rangle, i = 1,...,n - k - 1.$$
That is, the orthonormal system $\omega_{1},...,\omega_{n - k - 1}$ needs to be orthogonal to  $\mathcal{W} = \textbf{\textrm{span}}\{x_{k + 1} - x_{1},...,x_{k + 1} - x_{k}, x_{k + 1} - z_{0}\}$. Since $\mathcal{W}$ is spanned by $k + 1$ vectors it follows that $\dim\mathcal{W}\leq k + 1$ and thus $\dim\mathcal{W}^{\bot}\geq n - k - 1$. Thus, we can find an orthonormal system $\omega_{1},...,\omega_{n - k - 1}$ which is in $\mathcal{W}^{\bot}$. The corresponding parameters $\mu_{1},...,\mu_{n - k - 1}$ are then uniquely determined by equation (2.4).\\

\textbf{$\bullet$ Recovering $f$ on each $k + 1$ plane in $\mathcal{H}'(n, k + 1)$}: Let us take a $k + 1$ plane $\Sigma_{\omega}'$ in $\mathcal{H}'(n, k + 1)$ and show how $f$ can be recovered on $\Sigma_{\omega}'$. As we previously mentioned, we can assume that $\{x_{k + 1} - x_{1},...,x_{k + 1} - x_{k}\}$ is an orthonormal set. Hence, since the vectors $\omega_{1},...,\omega_{n - k - 1}$ also form an orthonormal set then from equation (2.3) it follows that $\{x_{k + 1} - x_{1},...,x_{k + 1} - x_{k}, \omega_{1},...,\omega_{n - k - 1}\}$ is an orthonormal set. Thus, there exists a rotation $A$ such that
$$\hskip-0.7cm A(x_{k + 1} - x_{1}) = e_{1},...,A(x_{k + 1} - x_{k}) = e_{k}, A\omega_{1} = e_{k + 2},...,A\omega_{n - k - 1} = e_{n}.$$
Let us choose a point $x_{0}\in\Bbb R^{n}\setminus\left(\mathbf{H}\{x_{1},...,x_{k + 1}\}\cup\mathbf{U}\{x_{1},...,x_{k + 1}\}\right)$ such that the $k$ plane $\Sigma_{x_{0}}$ is contained in the $k + 1$ plane $\Sigma_{\omega}'$. If we denote $x_{0}' = Ax_{0}, x_{k + 1}' = Ax_{k + 1}$ then condition (2.5) is equivalent to
\vskip-0.2cm
\begin{equation}\hskip-5cm\langle x_{0}', e_{i + k + 1}\rangle = \langle x_{k + 1}', e_{i + k + 1}\rangle, i = 1,..., n - k - 1.\end{equation}
The assumption that $x_{0}$ is not in the unique $k$ plane which passes through $x_{1},...,x_{k + 1}$ is equivalent to the assumption that $\textbf{\textrm{rank}}\{x_{k + 1} - x_{0}, x_{k + 1} - x_{1},...,x_{k + 1} - x_{k}\} = k + 1$ or equivalently that $\textbf{\textrm{rank}}\{x_{k + 1}' - x_{0}', e_{1},...,e_{k}\} = k + 1$ and from equation (2.6) this is equivalent to $x_{0, k + 1}'\neq x_{k + 1, k + 1}'$. The assumption that $x_{0}\notin\mathbf{U}\{x_{1},...,x_{k + 1}\}$ is equivalent to the assumption that there exists an index $1\leq i\leq k + 1$ such that $\langle x_{i} - x_{0}, x_{0}\rangle\neq0$ or equivalently that $\langle x_{i}' - x_{0}', x_{0}'\rangle\neq0$ where we denote $x_{i}' = Ax_{i}$.

Hence, the $k$ plane $A\Sigma_{x_{0}}$ is contained in the $k + 1$ plane $A\Sigma_{\omega}'$ if and only if $x_{0}' = Ax_{0}$ has the form
\vskip-0.2cm
\begin{equation}\hskip-4.75cm x_{0}' = (u, u_{0}), u_{0} =  (x_{k + 1,k + 2}',...,x_{k + 1, n}'), u\in\Bbb R^{k + 1}\end{equation}
and which satisfies the following condition
\vskip-0.2cm
$$\hskip-4.4cm\hskip0.1cm u_{k + 1}\neq x_{k + 1, k + 1}', \exists i, 1\leq i\leq k + 1, \langle x_{i}' - x_{0}', x_{0}'\rangle\neq0$$
or equivalently the condition
$$\hskip-2.1cm\textrm{(iii)}\hskip0.1cm u_{k + 1}\neq x_{k + 1, k + 1}', \exists i, 1\leq i\leq k + 1, \langle x_{i}' - (u, u_{0}), (u, u_{0})\rangle\neq0.$$
Let us recall again that the $k + 1$ plane $\Sigma_{\omega}'$ is given by the system of equations
$$\hskip-6.6cm\langle x, \omega_{i}\rangle = \langle x_{k + 1}, \omega_{i}\rangle, i = 1,...,n - k - 1$$
or equivalently by the system of equations
$$\hskip-4.7cm\langle Ax, e_{k + i + 1}\rangle = \langle x_{k + 1}', e_{k + i + 1}\rangle, i = 1,...,n - k - 1.$$
Hence, the $k + 1$ plane $A\Sigma_{\omega}'$ is given by
\vskip-0.2cm
$$\hskip-3.9cm A\Sigma_{\omega}' = \{(u,u_{0}):u\in\Bbb R^{k + 1}, u_{0} = (x_{k + 1,k + 2}',...,x_{k + 1, n}')\}$$
and for this plane we define its origin to be the point $(\overline{0},u_{0})$. Now, we claim that for each point $x_{0}'$ which is of the form (2.7) and satisfies condition (iii) the closest point of the $k$ plane $A\Sigma_{x_{0}}$ to the origin in $A\Sigma_{\omega}'$ is obtained at $x_{0}'$. Indeed, the closest point of $\Sigma_{x_{0}}$ to the origin is $x_{0}$ (since $\Sigma_{x_{0}}\subset\mathbf{H}_{x_{0}}$) and thus the closest point of $A\Sigma_{x_{0}}$ to the origin is $x_{0}' = Ax_{0}$. Now, if there exists a point $x_{0}''\in A\Sigma_{x_{0}}\subset A\Sigma_{\omega}'$ whose distance to the origin $(\overline{0},u_{0})$ in $A\Sigma_{\omega}'$ is smaller than that of $x_{0}'$ then we can assume that
$$\hskip-6.3cm x_{0}' = (u, u_{0}), x_{0}'' = (u'', u_{0}), u, u''\in\Bbb R^{k + 1}$$
and where $\left|u''\right| < \left|u\right|$. This will imply that $\left|x_{0}''\right| < \left|x_{0}'\right|$ an so $x_{0}''$ is a point in $A\Sigma_{x_{0}}$ whose distance to the origin in $\Bbb R^{n}$ is smaller than that of $x_{0}'$. This is obviously a contradiction.

Hence, for a point $x_{0}$ such that $x_{0}' = Ax_{0}$ is of the form $(u, u_{0})$, where $u$ satisfies condition (iii), we have that $A\Sigma_{x_{0}}\subset A\Sigma_{\omega}'$ and that its closest point to the origin in $A\Sigma_{\omega}'$ is obtained at $x_{0}'$. This means that if we project the $k + 1$ plane $\Sigma_{\omega}' = \Bbb R^{k + 1}\times\{u_{0}\}$ to $\Bbb R^{k + 1}$ then the closest point, of the projection of $A\Sigma_{x_{0}}$, to the origin $\overline{0}$ in $\Bbb R^{k + 1}$ is obtained at the point $u$.

Thus, it follows that the projection of $A\Sigma_{x_{0}}$ to $\Bbb R^{k + 1}$ is the following $k$ dimensional hyperplane
\vskip-0.2cm
\begin{equation}\hskip-7.5cm \Pi_{u} = \left\{\langle x' - u, u \rangle = 0: x'\in\Bbb R^{k + 1}\right\}.\end{equation}
Now, the restriction that the point $u$ must satisfy condition (iii) can be omitted since the set of points $u$ which do not satisfy this condition is of codimension $1$ in $\Bbb R^{k + 1}$. Thus, by passing to limits when preforming integration we can obviously obtain the integrals on any hyperplane of the form (2.8). Thus, we can assume that $u\in\Bbb R^{k + 1}$. Hence, if $u = r\sigma$ where $r\geq0, \sigma\in\Bbb S^{k}$ then the $k$ plane $\Pi_{u}$ is given by
\vskip-0.2cm
$$\hskip-8cm \Pi_{r, \sigma}' = \left\{\langle x', \sigma\rangle = r: x'\in\Bbb R^{k + 1}\right\}.$$
Observe that every $k$ dimensional hyperplane in $\Bbb R^{k + 1}$ can be written as $\Pi_{r, \sigma}'$ for some fixed $r\geq0, \sigma\in\Bbb S^{k}$.
Every function in $C_{0}^{\infty}(\Bbb R^{k + 1})$ which decays to zero faster than $|x|^{-N}$, $N > k + 1$, at infinity can be recovered from its classical Radon transform (see \cite[Chap. 1.3, Theorem 3.1]{9}) which integrates the function on every $k$ dimensional hyperplane in $\Bbb R^{k + 1}$. From our assumption on the function $f$ it follows that the restriction of $g = f\circ A^{-1}$ to the $k + 1$ plane $A\Sigma_{\omega}'$ satisfies these conditions and thus it can be recovered on this plane. Equivalently, this means that $f$ can be recovered on $\Sigma_{\omega}'$. Thus, Theorem 2.1 is proved.
\end{proof}

\subsection{The case of integration on lines in $\Bbb R^{3}$ with a fixed distance from a given point}

\hskip0.6cm For $r > 0$, a point $p$ in $\Bbb R^{3}$ and a compactly supported continuous function $f$, defined on $\Bbb R^{3}$, we consider the following problem. Suppose that the integrals of $f$ are given on each line whose distance from the point $p$ is equal to $r$ and we would like to recover $f$ from this family $\Pi$ of lines. Since the intersection of the interior of the sphere $\mathrm{S}' = \mathrm{S}(p, r)$ with any line in $\Pi$ is empty then obviously $f$ cannot be recovered inside $\mathrm{S}'$ and thus we can assume that each function in question vanishes inside $\mathrm{S}'$. Without loss of generality we can assume that $p = \overline{0}, r = 1$. That is, our family $\Pi$ consists of all the lines which are tangent to the unit sphere $\Bbb S^{2}$ and our aim is to recover a continuous compactly supported function, which vanishes inside $\Bbb S^{2}$, given its integrals on each line in $\Pi$. Since $\Bbb S^{2}$ is of dimension two and since, for each point $p\in\Bbb S^{2}$, the family of lines which are tangent to $\Bbb S^{2}$ and pass through $p$ is one dimensional it follows that $\dim\Pi = 3$. Hence, our problem is well-posed.

Observe that for a point $x\in\Bbb R^{3}$, satisfying $|x|\geq1$, the family of all lines which are tangent to the unit sphere $\Bbb S^{2}$ and pass through $x$ is a cone with the apex at $x$ which is tangent to $\Bbb S^{2}$ where the tangency set is a circle on $\Bbb S^{2}$. Let us denote by $C_{x}$ the above considered cone and denote by $\Lambda$ the family of all such cones which are obtained by taking points $x\in\Bbb R^{3}$ satisfying $|x|\geq1$. Observe that if $x = \lambda e_{3}$ where $\lambda\geq1$ then the family $\Lambda$ can be parameterized as follows
$$\hskip-7.3cm\Lambda = \{A^{-1}C_{\lambda e_{3}}:A\in SO(2), \lambda\geq1\}$$
where $SO(2)$ denotes the group of rotations in $\Bbb R^{3}$. For every $\lambda\geq1$ the cone $C_{\lambda e_{3}}$ can be parameterized as follows
\begin{equation}\hskip-2.4cm C_{\lambda e_{3}} = \left\{\left(t\cos\phi, t\sin\phi, \lambda + t\sqrt{\lambda^{2} - 1}\right): (\phi, t)\in[-\pi,\pi)\times\Bbb R\right\}.\end{equation}

Now, let $f$ be a continuous function, defined in $\Bbb R^{3}$, with compact support. Then, our data consists of the following integrals
\begin{equation}\hskip-1.35cm G(A, \lambda) = \int_{A^{-1}C_{\lambda e_{3}}}f(x)dS_{x} = \int_{C_{\lambda e_{3}}}f\left(A^{-1}x\right)dS_{x}, A\in SO(2), \lambda\geq1.\end{equation}

\begin{remark}

Observe that with the parametrization (2.9) of the cone $C_{\lambda e_{3}}$ we have that its infinitesimal area measure is given by $dC_{\lambda e_{3}} = \lambda|t|d\phi dt$. However, in our case the expression $dS_{x}$ given in equation (2.10) is chosen such that $dS_{x} = \lambda d\phi dt$ (i.e., the term $|t|$ is omitted). We can choose $dS_{x}$ in this way since for every $\lambda\geq1$ we know the factorization of $C_{\lambda e_{3}}$ into lines passing through its apex. Each such line, given by the parametrization (2.9), where the parameter $\phi$ is fixed, has infinitesimal length measure given by $\lambda dt$. Thus, we can just integrate on this family of lines with respect to $\phi$ with the infinitesimal angular measure $d\phi$.

\end{remark}

Our aim is to recover $f$ from the family of integrals given by equation (2.10). The inversion formula is given in Theorem 2.3 below. Before formulating and proving Theorem 2.3 we introduce some notations and definitions.
\vskip0.2cm
For every integer $m\geq0$ denote by $C_{m}^{\frac{1}{2}}$ the Gegenbauer polynomial of order $1 / 2$ and degree $m$.

Let $T_{A}, A\in SO(2)$ be the quasi-regular representation of the group $SO(2)$. That is, for a function $g\in L^{2}\left(\Bbb S^{2}\right)$ the operator $T$ acts on $g$ by $T_{A}(g) = g\circ A^{-1}$. Let $Y_{k}^{m}, m\geq0, |k|\leq m$ be an orthonormal complete system of spherical harmonics in $\Bbb S^{2}$ and let $\tau_{k,p}^{m}, m\geq0, |k|,|p|\leq m$ be the components of the representation matrix of $T$ with respect to this family of spherical harmonics. That is
\vskip-0.2cm
$$\hskip-6.7cm Y_{k}^{m}\left(A^{- 1}\omega\right) = \sum_{p = -m}^{m}\tau_{k,p}^{m}(A)Y_{p}^{m}(\omega).$$
The functions $\tau_{k,p}^{m}$ satisfy the following orthogonality relations (see \cite[Chap. 1.5.2]{16}).
\[\hskip-2.5cm\int_{SO(2)}\tau_{k,p}^{m}(A)\overline{\tau_{k',p'}^{m'}(A)}dA =
\begin{cases}
0, \hskip0.6cm(m,k,p)\neq(m',k',p')\\
\frac{1}{2m + 1}, \textrm{o.w}.
\end{cases}\]
Now we can formulate Theorem 2.3.

\begin{thm}
Let $f$ be a continuous function, defined in $\Bbb R^{3}$, which has compact support and vanishes inside $\Bbb S^{2}$ and let
\vskip-0.2cm
$$\hskip-7cm f(r\omega) = \sum_{m = 0}^{\infty}\sum_{k = - m}^{m}f_{m,k}(r)Y_{k}^{m}(\omega)$$
be the expansion of $f$ into spherical harmonics. Denote
\[\hskip1cm\widetilde{G}_{m,k,p'}(\alpha) = \frac{c_{m}}{Y_{p'}^{m}(e_{3})}\cdot G_{m,k,p'}\left(\frac{1}{\cos\alpha}\right), |\alpha|\leq\frac{\pi}{2}, g_{m,k}(\alpha) =
\begin{cases}
\widetilde{f}_{m,k}(\alpha), \hskip0.15cm|\alpha| \leq \frac{\pi}{2},\\
0, \hskip0.45cm\frac{\pi}{2} < |\alpha|\leq\pi,
\end{cases}
\]
where
$$\hskip-0.9cm G_{m,k,p}(\lambda) = \int_{SO(2)}G(A, \lambda)\overline{\tau_{m,k}^{p}(A)}dA, \widetilde{f}_{m,k}(\alpha) = \frac{1}{\cos^{2}\alpha}f_{m,k}\left(\frac{1}{\cos\alpha}\right),$$
$c_{m} = (2m + 1)C_{m}^{\frac{1}{2}}(1) / 2\pi$ and where $\widetilde{G}_{m,k,p'}$ is defined on the whole interval $[-\pi,\pi]$ by defining
\vskip-0.2cm
$$\hskip-4.5cm\widetilde{G}_{m,k,p'}(\alpha) = (- 1)^{m}\widetilde{G}_{m,k,p'}(\pi - |\alpha|), \frac{\pi}{2}\leq|\alpha|\leq\pi.$$
Then, for the following Fourier expansions on the interval $[-\pi,\pi]$
$$\hskip0.7cm\widetilde{G}_{m,k,p'}(\alpha) = \sum_{n = -\infty}^{\infty}c_{1,n}e^{in\alpha}, g_{m,k}(\alpha) = \sum_{n = -\infty}^{\infty}c_{2,n}e^{in\alpha}, C_{m}^{\frac{1}{2}}\left(\cos\alpha\right) = \sum_{n = -\infty}^{\infty}c_{3,n}e^{in\alpha}$$
we have that $c_{2,n} = c_{1,n} / \left(2\pi\cdot c_{3,n}\right)$.
\end{thm}

\begin{remark}

In Theorem 2.3 the parameter $p', |p'|\leq m$ can be any parameter which satisfies $Y_{p'}^{m}(e_{3})\neq0$. Observe that by Theorem 2.3 we can recover the function $g_{m,k}$ and thus we can also recover the function $\widetilde{f}_{m,k}$ which is equivalent of recovering $f_{m,k}$ on $[1,\infty)$. Since, by our assumption, $f$ vanishes inside $\Bbb S^{2}$ it follows that $f_{m,k}$ vanishes in $[0,1]$ and thus we can extract $f_{m,k}$ on the whole line $\Bbb R$. Hence, since we can obtain $f_{m,k}$ for every $m\geq0, |k|\leq m$ then obviously the function $f$ can be recovered. Also, observe that in the formulation of Theorem 2.3 we implicitly assumed that $f$ satisfies the following condition 
$$\hskip-3cm\textrm{if}\hskip0.2cm g_{m,k}(\alpha) = \sum_{n = -\infty}^{\infty}c_{2,n}e^{in\alpha}\hskip0.2cm\mathrm{and}\hskip0.2cm C_{m}^{\frac{1}{2}}\left(\cos\alpha\right) = \sum_{n = -\infty}^{\infty}c_{3,n}e^{in\alpha}$$
$$\hskip-9cm \textrm{then}\hskip0.2cm c_{2,n}\neq0\Rightarrow c_{3,n}\neq0$$
so that in the equation $c_{2,n} = c_{1,n} / \left(2\pi\cdot c_{3,n}\right)$ the division of the right hand side by $c_{3,n}$ is justified. If this condition is not satisfied then, as the proof of Theorem 2.3 will show, uniqueness does not occur for the function $f$ (i.e, the set of integrals of $f$ on lines which are tangent to the unit sphere is not enough in order to reconstruct $f$). For example, if $f$ is a nonzero radial function which vanishes inside the unit sphere and has compact support then our data consists of the integral of $f$ on just one line with unit distance from the origin. But obviously this is not enough to reconstruct $f$ and indeed it can be easily shown that $f$ does not satisfy the last condition.

\end{remark}

\vskip0.2cm

\textbf{Proof of Theorem 2.3}: For every $A\in SO(2)$ we have
$$\hskip-3.8cm f\left(A^{-1}r\omega\right) = f\left(rA^{-1}\omega\right) = \sum_{m = 0}^{\infty}\sum_{k = - m}^{m}f_{m,k}(r)Y_{k}^{m}\left(A^{-1}\omega\right)$$
$$\hskip-6.1cm = \sum_{m = 0}^{\infty}\sum_{k = - m}^{m}f_{m,k}(r)\sum_{p = - m }^{m}\tau_{m, k}^{p}(A)Y_{p}^{m}\left(\omega\right).$$
Using the orthogonality relations for the family of functions $\tau_{m, k}^{p}, m\geq0, |k|, |p|\leq m$ on $SO(2)$ we have that
\vskip-0.2cm
$$\hskip-3.5cm\int_{SO(2)}f\left(A^{-1}r\omega\right)\overline{\tau_{m, k}^{p}(A)}dA = \frac{1}{2m + 1}\cdot f_{m,k}(r)Y_{p}^{m}(\omega).$$
Hence, we obtained that
$$\hskip-1.5cm\int_{SO(2)}G(A, \lambda)\overline{\tau_{m,k}^{p}(A)}dA = \frac{1}{2m + 1}\int_{C_{\lambda e_{3}}}f_{m,k}(|x|)Y_{p}^{m}\left(\frac{x}{|x|}\right)dS_{x}$$
where $\lambda\geq1$ and $|p|\leq m$. Using the parametrization (2.9) of the cone $C_{\lambda e_{3}}$ we have
\vskip-0.2cm
$$\hskip-6.5cm G_{m,k,p}(\lambda) = \int_{SO(2)}G(A, \lambda)\overline{\tau_{m,k}^{p}(A)}dA $$
$$\hskip0.2cm = \frac{\lambda}{2m + 1}\int_{-\infty}^{\infty}\int_{-\pi}^{\pi}f_{m,k}\left(\sqrt{\lambda^{2}\left(1 + t^{2}\right) + 2t\lambda\sqrt{\lambda^{2} - 1}}\right)$$ \begin{equation}\hskip1cm\times Y_{p}^{m}\left(\frac{t(\cos\phi, \sin\phi, 0) + \left(\lambda + t\sqrt{\lambda^{2} - 1}\right)e_{3}}{\sqrt{\lambda^{2}\left(1 + t^{2}\right) + 2t\lambda\sqrt{\lambda^{2} - 1}}}\right)d\phi dt, \lambda\geq1.\end{equation}
Let $\alpha$ be a point in the interval $[-\pi,\pi)$ such that
$$\cos\alpha = \frac{\lambda + t\sqrt{\lambda^{2} - 1}}{\sqrt{\lambda^{2}\left(1 + t^{2}\right) + 2t\lambda\sqrt{\lambda^{2} - 1}}}, \sin\alpha = \frac{t}{\sqrt{\lambda^{2}\left(1 + t^{2}\right) + 2t\lambda\sqrt{\lambda^{2} - 1}}}.$$
Every spherical harmonic in $\Bbb S^{2}$ satisfies the following identity (see for example \cite{1})
\begin{equation}
\hskip-2cm\int_{\Bbb S_{\psi}^{1}}Y_{k}^{m}((\cos\alpha)\psi + (\sin\alpha)\sigma)d\sigma = (2\pi / C_{m}^{\frac{1}{2}}(1))\cdot C_{m}^{\frac{1}{2}}(\cos\alpha)Y_{k}^{m}(\psi)
\end{equation}
where $\psi\in\Bbb S^{2}$. Hence, using equations (2.11)-(2.12) and the definition of the parameter $c_{m}$ we have
$$\hskip-7.5cm G_{m,k,p}(\lambda) = \int_{SO(2)}G(A, \lambda)\overline{\tau_{m,k}^{p}(A)}dA$$
$$\hskip-1.5cm = \frac{\lambda}{c_{m}}\cdot Y_{p}^{m}(e_{3})\int_{-\infty}^{\infty}f_{m,k}\left(\sqrt{\lambda^{2}\left(1 + t^{2}\right) + 2t\lambda\sqrt{\lambda^{2} - 1}}\right)$$ \begin{equation}\hskip3cm\times C_{m}^{\frac{1}{2}}\left(\frac{\lambda + t\sqrt{\lambda^{2} - 1}}{\sqrt{\lambda^{2}\left(1 + t^{2}\right) + 2t\lambda\sqrt{\lambda^{2} - 1}}}\right)dt, \lambda\geq1, |p|\leq m.\end{equation}
Now, observe that there exists an integer $p$ satisfying $|p|\leq m$ such that $Y_{p}^{m}(e_{3})\neq0$. Otherwise we can just choose a spherical harmonic $Y^{m}$ of degree $m$ and choose a point $\omega_{0}\in\Bbb S^{2}$ such that $Y^{m}(\omega_{0})\neq0$. Then, if $A$ is a rotation such that $A\omega_{0} = e_{3}$ then $K(\omega) = Y^{m}\left(A^{-1}\omega\right)$ will also be a spherical harmonic of degree $m$ such that $K(e_{3})\neq0$. Now, since $K$ can be expanded into spherical harmonics from the family $Y_{p}^{m}, |p|\leq m$ then if $Y_{p}^{m}(e_{3}) = 0$ for every $|p|\leq m$ then we will also have $K(e_{3}) = 0$ which is a contradiction.

Let us choose $p', |p'|\leq m$ such that $Y_{p'}^{m}(e_{3})\neq0$. Then, from equation (2.13) we have
\vskip-0.2cm
$$\hskip-10.5cm \frac{c_{m}}{\lambda}\cdot G_{m,k,p'}(\lambda)/ Y_{p'}^{m}(e_{3})$$
$$\hskip-0.5cm = \int_{-\infty}^{\infty}f_{m,k}\left(\sqrt{\lambda^{2}\left(1 + t^{2}\right) + 2t\lambda\sqrt{\lambda^{2} - 1}}\right) C_{m}^{\frac{1}{2}}\left(\frac{\lambda + t\sqrt{\lambda^{2} - 1}}{\sqrt{\lambda^{2}\left(1 + t^{2}\right) + 2t\lambda\sqrt{\lambda^{2} - 1}}}\right)dt$$
$$ = \int_{-\infty}^{-\frac{\sqrt{\lambda^{2} - 1}}{\lambda}}f_{m,k}\left(\sqrt{\lambda^{2}\left(1 + t^{2}\right) + 2t\lambda\sqrt{\lambda^{2} - 1}}\right) C_{m}^{\frac{1}{2}}\left(\frac{\lambda + t\sqrt{\lambda^{2} - 1}}{\sqrt{\lambda^{2}\left(1 + t^{2}\right) + 2t\lambda\sqrt{\lambda^{2} - 1}}}\right)dt$$
\begin{equation} + \int_{-\frac{\sqrt{\lambda^{2} - 1}}{\lambda}}^{\infty}f_{m,k}\left(\sqrt{\lambda^{2}\left(1 + t^{2}\right) + 2t\lambda\sqrt{\lambda^{2} - 1}}\right) C_{m}^{\frac{1}{2}}\left(\frac{\lambda + t\sqrt{\lambda^{2} - 1}}{\sqrt{\lambda^{2}\left(1 + t^{2}\right) + 2t\lambda\sqrt{\lambda^{2} - 1}}}\right)dt.\end{equation}
Observe that the function
\vskip-0.2cm
$$\hskip-6cm u(t) = \sqrt{\lambda^{2}\left(1 + t^{2}\right) + 2t\lambda\sqrt{\lambda^{2} - 1}}, t\in\Bbb R$$
is injective on the intervals $\left(-\infty, -\sqrt{\lambda^{2} - 1} / \lambda\right]$ and $\left[-\sqrt{\lambda^{2} - 1} / \lambda, \infty\right)$. Hence, we can make the following change of variables
\begin{equation}\hskip-1.5cm x = \sqrt{\lambda^{2}\left(1 + t^{2}\right) + 2t\lambda\sqrt{\lambda^{2} - 1}}, dx = \frac{\lambda\left(t\lambda + \sqrt{\lambda^{2} - 1}\right)dt}{\sqrt{\lambda^{2}\left(1 + t^{2}\right) + 2t\lambda\sqrt{\lambda^{2} - 1}}}.\end{equation}
Extracting the variable $t$ from equation (2.15) we have
$$\hskip-3.2cm t = -(\sqrt{\lambda^{2} - 1} + \sqrt{x^{2} - 1}) / \lambda, t\in\left(-\infty, -\sqrt{\lambda^{2} - 1} / \lambda\right],$$
$$\hskip-3.5cm t = -(\sqrt{\lambda^{2} - 1} - \sqrt{x^{2} - 1}) / \lambda, t\in\left[-\sqrt{\lambda^{2} - 1} / \lambda, \infty\right).$$
In both cases we have $dx = \lambda\sqrt{x^{2} - 1}dt / x$. Thus
\vskip-0.2cm
$$\hskip-8.5cm c_{m}\cdot G_{m,k,p'}(\lambda)/Y_{p'}^{m}(e_{3})$$
$$\hskip-3.5cm = \int_{1}^{\infty}f_{m,k}(x)C_{m}^{\frac{1}{2}}\left(\frac{1 - \sqrt{\lambda^{2} - 1}\sqrt{x^{2} - 1}}{\lambda x}\right)\frac{xdx}{\sqrt{x^{2} - 1}}.$$ \begin{equation}\hskip-2.5cm + \int_{1}^{\infty}f_{m,k}(x)C_{m}^{\frac{1}{2}}\left(\frac{1 + \sqrt{\lambda^{2} - 1}\sqrt{x^{2} - 1}}{\lambda x}\right)\frac{xdx}{\sqrt{x^{2} - 1}}, \lambda\geq1.\end{equation}
Denoting $\lambda' = 1 / \lambda$ and making the change of variables $x = 1 / x'$ we have
$$\hskip-8cm c_{m}\cdot G_{m,k,p'}\left(\frac{1}{\lambda'}\right) / Y_{p'}^{m}(e_{3})$$
$$\hskip-1cm = \int_{0}^{1}f_{m,k}\left(\frac{1}{x'}\right)C_{m}^{\frac{1}{2}}\left(\lambda'x' - \sqrt{1 - (\lambda')^{2}}\sqrt{1 - (x')^{2}}\right)\frac{dx'}{(x')^{2}\sqrt{1 - (x')^{2}}}$$ \begin{equation} + \int_{0}^{1}f_{m,k}\left(\frac{1}{x'}\right)C_{m}^{\frac{1}{2}}\left(\lambda'x' + \sqrt{1 - (\lambda')^{2}}\sqrt{1 - (x')^{2}}\right)\frac{dx'}{(x')^{2}\sqrt{1 - (x')^{2}}}\end{equation}
where $0\leq\lambda'\leq1$. Observe that both integrals in equation (2.17) converge since for every $m\geq0, |k|\leq m$ the function $f_{m,k}$ vanishes at infinity (since, by assumption, $f$ is compactly supported). Hence, there is no singularity near $x' = 0$.

Let us denote $\lambda' = \cos\alpha, 0\leq\alpha\leq\pi / 2$ and make the change of variables $x' = \cos\beta$. Then, from equation (2.17) we have\\
$$\hskip-6.5cm c_{m}\cdot G_{m,k,p'}\left(\frac{1}{\cos\alpha}\right) / Y_{p'}^{m}(e_{3})$$
$$\hskip-2.5cm = \int_{0}^{\frac{\pi}{2}}f_{m,k}\left(\frac{1}{\cos\beta}\right)C_{m}^{\frac{1}{2}}\left(\cos\alpha\cos\beta - \sin\alpha\sin\beta\right)\frac{d\beta}{\cos^{2}\beta}$$ $$\hskip-2cm + \int_{0}^{\frac{\pi}{2}}f_{m,k}\left(\frac{1}{\cos\beta}\right)C_{m}^{\frac{1}{2}}\left(\cos\alpha\cos\beta + \sin\alpha\sin\beta\right)\frac{d\beta}{\cos^{2}\beta}$$
$$\hskip-0.5cm = \int_{0}^{\frac{\pi}{2}}\widetilde{f}_{m,k}(\beta)C_{m}^{\frac{1}{2}}\left(\cos(\alpha + \beta)\right)d\beta +  \int_{0}^{\frac{\pi}{2}}\widetilde{f}_{m,k}(\beta)C_{m}^{\frac{1}{2}}\left(\cos(\alpha - \beta)\right)d\beta$$
\begin{equation}\hskip-2.5cm = \int_{0}^{\frac{\pi}{2}}\widetilde{f}_{m,k}(\beta)\left(C_{m}^{\frac{1}{2}}\left(\cos(\alpha + \beta)\right) + C_{m}^{\frac{1}{2}}\left(\cos(\alpha - \beta)\right)\right)d\beta\end{equation}
where we define
\vskip-0.2cm
$$\hskip-6.5cm\widetilde{f}_{m,k}(\beta) = \frac{1}{\cos^{2}\beta}f_{m,k}\left(\frac{1}{\cos\beta}\right).$$
Observe that $\widetilde{f}_{m,k}(\beta)$ is an even function defined on $-\pi / 2\leq \beta\leq \pi / 2$. Since the sum of the Gegenbauer polynomials in the integral (2.18) is an even function of $\beta$ it follows that
\vskip-0.2cm
$$\hskip-6.3cm c_{m}\cdot G_{m,k,p'}\left(\frac{1}{\cos\alpha}\right) / Y_{p'}^{m}(e_{3})$$
$$\hskip-3cm = \frac{1}{2}\int_{-\frac{\pi}{2}}^{\frac{\pi}{2}}\widetilde{f}_{m,k}(\beta)\left(C_{m}^{\frac{1}{2}}\left(\cos(\alpha + \beta)\right) + C_{m}^{\frac{1}{2}}\left(\cos(\alpha - \beta)\right)\right)d\beta$$
$$\hskip-0.2cm = \frac{1}{2}\left[\underset{\mathcal{I}_{1}}{\underbrace{\int_{-\frac{\pi}{2}}^{\frac{\pi}{2}}\widetilde{f}_{m,k}(\beta)C_{m}^{\frac{1}{2}}\left(\cos(\alpha + \beta)\right)d\beta}} + \underset{\mathcal{I}_{2}}{\underbrace{\int_{-\frac{\pi}{2}}^{\frac{\pi}{2}}\widetilde{f}_{m,k}(\beta)C_{m}^{\frac{1}{2}}\left(\cos(\alpha - \beta)\right)d\beta}}\right].$$
Since $\widetilde{f}_{m,k}(\beta)$ is an even function of $\beta$ it follows easily that $\mathcal{I}_{1} = \mathcal{I}_{2}$. Thus,
\begin{equation}\hskip-4.6cm \widetilde{G}_{m,k,p'}(\alpha) = \int_{-\frac{\pi}{2}}^{\frac{\pi}{2}}\widetilde{f}_{m,k}(\beta)C_{m}^{\frac{1}{2}}\left(\cos(\alpha - \beta)\right)d\beta\end{equation}
where we denote
\vskip-0.2cm
$$\hskip-4.5cm \widetilde{G}_{m,k,p'}(\alpha) = c_{m}\cdot G_{m,k,p'}\left(\frac{1}{\cos\alpha}\right) / Y_{p'}^{m}(e_{3}).$$
Observe that the right hand side of equation (2.19) is given for $0\leq\alpha\leq\pi / 2$, but using the change of variables $\beta\mapsto-\beta$ and using the evenness of $\widetilde{f}_{m,k}$ it follows that the right hand side of equation (2.19) is an even function of $\alpha$. Hence, $\widetilde{G}_{m,k,p}$ can be extended, as an even function of $\alpha$, to the interval $[-\pi / 2, \pi / 2]$. The right hand side of equation (2.19) is defined for all $\alpha\in[-\pi,\pi]$ and if we denote it by $\mathcal{K}_{m,k}$ then we have that
\vskip-0.2cm
$$\hskip-4.9cm\mathcal{K}_{m,k}(\alpha) = (- 1)^{m}\mathcal{K}_{m,k}(\pi - |\alpha|), \frac{\pi}{2}\leq|\alpha|\leq\pi.$$
Indeed, let $\alpha$ satisfy $\pi / 2\leq|\alpha|\leq\pi$, then we have
$$\hskip-4.2cm\mathcal{K}_{m,k}(\alpha) = \int_{-\frac{\pi}{2}}^{\frac{\pi}{2}}\widetilde{f}_{m,k}(\beta)C_{m}^{\frac{1}{2}}\left(\cos(\alpha - \beta)\right)d\beta$$
$$\hskip-4.5cm = \int_{-\frac{\pi}{2}}^{\frac{\pi}{2}}\widetilde{f}_{m,k}(\beta)C_{m}^{\frac{1}{2}}\left(-\cos(\pi \pm\alpha \mp \beta)\right)d\beta$$
$$\hskip-4cm = (- 1)^{m}\int_{-\frac{\pi}{2}}^{\frac{\pi}{2}}\widetilde{f}_{m,k}(\beta)C_{m}^{\frac{1}{2}}\left(\cos(\pi \pm\alpha \mp \beta)\right)d\beta.$$
Hence, in case where $\pi / 2\leq\alpha\leq\pi$ then we choose the sign
$$\hskip-6.5cm\pi \pm\alpha \mp \beta = \pi - \alpha + \beta = \pi - |\alpha| + \beta$$
and then by making the change of variables $\beta\mapsto -\beta$ and using the evenness of $\widetilde{f}_{m,k}$ we have
$$\mathcal{K}_{m,k}(\alpha) = (- 1)^{m}\int_{-\frac{\pi}{2}}^{\frac{\pi}{2}}\widetilde{f}_{m,k}(\beta)C_{m}^{\frac{1}{2}}\left(\cos(\pi - |\alpha| - \beta)\right)d\beta = (- 1)^{m}\mathcal{K}_{m,k}(\pi - |\alpha|).$$
In case where $-\pi\leq\alpha\leq-\pi / 2$ then we choose the sign
$$\hskip-6.5cm\pi \pm\alpha \mp \beta = \pi + \alpha - \beta = \pi - |\alpha| - \beta$$
and then we can prove in the exact same way that $\mathcal{K}_{m,k}(\alpha) = (- 1)^{m}\mathcal{K}_{m,k}(\pi - |\alpha|)$. Hence, if we extend the  function $\widetilde{G}_{m,k,p'}$ to the larger domain $[-\pi,\pi]$ by defining
$$\hskip-4.5cm\widetilde{G}_{m,k,p'}(\alpha) = (- 1)^{m}\widetilde{G}_{m,k,p'}(\pi - |\alpha|), \frac{\pi}{2}\leq|\alpha|\leq\pi$$
then equation (2.19) is valid for all $\alpha$ in $[-\pi,\pi]$. Let us define the function $g_{m,k}$ as follows
\vskip-0.2cm
\[\hskip-7cm g_{m,k}(\beta) =
\begin{cases}
\widetilde{f}_{m,k}(\beta), \hskip0.9cm|\beta| \leq \frac{\pi}{2},\\
0, \hskip1.2cm\frac{\pi}{2} < |\beta|\leq\pi.
\end{cases}
\]
Then, from equation (2.19) we have
\begin{equation}\hskip-3.3cm\widetilde{G}_{m,k,p'}(\alpha) = \int_{-\pi}^{\pi}g_{m,k}(\beta)C_{m}^{\frac{1}{2}}\left(\cos(\alpha - \beta)\right)d\beta, \alpha\in[-\pi,\pi].\end{equation}
Assume now that $\widetilde{G}_{m,k,p'}$ and $g_{m,k}$ are defined on the whole line $\Bbb R$ by taking their $2\pi$ periodic extensions. Bearing in mind that the convolution $h_{3}$ of two $2\pi$ periodic functions $h_{1}$ and $h_{2}$ is defined by
\vskip-0.2cm
$$\hskip-5.7cm h_{3}(\alpha) = \frac{1}{2\pi}\int_{-\pi}^{\pi}h_{1}(\beta)h_{2}(\alpha - \beta)d\beta, \alpha\in\Bbb R$$
and that the convolution theorem for this type of functions asserts that if
$$\hskip-1cm h_{1}(\alpha) = \sum_{n = -\infty}^{\infty}c_{1,n}e^{in\alpha}, h_{2}(\alpha) = \sum_{n = -\infty}^{\infty}c_{2,n}e^{in\alpha}, h_{3}(\alpha) = \sum_{n = -\infty}^{\infty}c_{3,n}e^{in\alpha}$$
then $c_{3,n} = c_{1,n}\cdot c_{2,n}$, Theorem 2.3 is obtained from equation (2.20) and by using the expansions of $\widetilde{G}_{m,k,p'}, g_{m,k}$ and $C_{m}^{\frac{1}{2}}(\cos(\cdot))$ into their corresponding Fourier series.

$\hskip14.4cm\square$

\subsection{The case of integration on lines in $\Bbb R^{3}$ with equal distances from two given points}

\hskip0.6cm Let $x$ and $y$ be two distinct points in $\Bbb R^{3}$ and consider the family $\Pi$ of all lines which have equal distances from these points. Without loss of generality we can assume that $x = e_{3}, y = -e_{3}$. In this case it can be easily observed that $\Pi$ is the set of all lines whose distance to the origin is obtained at their intersection point with the $XY$ plane, when considering lines which are not contained in this plane, and also all the lines which are contained in the $XY$ plane. In other words, for each point $p = (p_{1}, p_{2}, 0)\neq\overline{0}$ in the $XY$ plane we consider the unique hyperplane $\mathbf{H}_{p}$ which passes through $p$ and whose normal is in the direction of $\overrightarrow{op}$ and then we take all the lines which pass through $p$ and are contained in $\mathbf{H}_{p}$. Since the family of hyperplanes $\mathbf{H}_{p}$ where $p\in\left(\Bbb R^{2}\setminus\{(0,0)\}\right)\times\{0\}$ is two dimensional and the set of all lines which pass through $p$ and are contained in $\mathbf{H}_{p}$ is one dimensional it follows that $\dim\Pi = 3$ (in case where $p = \overline{0}$ then we can take all the lines which pass through $p$).

Our aim in this section is to recover a continuous function $f$, defined in $\Bbb R^{3}$, in case where we are given the integrals of $f$ on the set of all lines in $\Pi$. Since $\dim\Pi = 3$ our problem is well posed. In order to guarantee that inversion formulas are obtainable, when considering integration on lines in $\Pi$, we must restrict our family of functions since any function $f = f(x_{1}, x_{2}, x_{3})$ which is radial with respect to the variable $x^{\ast} = (x_{1}, x_{2})$ and is odd with respect to the variable $x_{3}$ produces no signals. However, if we assume that $f$ is compactly supported in the half space $x_{3} > 0$ then we can obtain an inversion formula as will be proved in Theorem 2.5 below.

Before formulating Theorem 2.5 we will need first to parameterize the family of lines $\Pi$ and introduce the Mellin transform. Observe that since $\Pi$ is invariant with respect to rotations which leave the vector $e_{3}$ fixed it follows that it is enough to parameterize the subset of lines in $\Pi$ which are parallel to the $YZ$ plane (and then we take rotations of these lines with respect to the $Z$ axis to obtain the whole set $\Pi$). If a line $l$ in this subset has a distance $\lambda\geq0$ from the origin and its projection to the $YZ$ plane forms an angle $\theta, -\pi / 2\leq\theta < \pi / 2$ with the $Y$ axis then $l$ has the following parametrization
\vskip-0.2cm
$$\hskip-4cm l_{\lambda,\theta} = \{(\lambda, t\cos\theta, t\sin\theta), t\in\Bbb R\}, \lambda\geq0, -\pi / 2\leq\theta < \pi / 2.$$
Hence, we obtain the following parametrization
\begin{equation}\Pi = \left\{l_{\lambda, \theta, \phi} =
\left\{\left(\begin{array}{c}
\lambda\cos\phi - t\cos\theta\sin\phi\\
\lambda\sin\phi + t\cos\theta\cos\phi\\
t\sin\theta
\end{array}\right), t\in\Bbb R\right\}:\lambda\geq0, -\frac{\pi}{2}\leq\theta < \frac{\pi}{2}, -\pi\leq\phi < \pi\right\}\end{equation}
for the family $\Pi$. If the function $G = G(\lambda, \theta ,\phi)$ is defined by
\begin{equation}\hskip-2.5cm G(\lambda, \theta, \phi) = \int_{_{l_{\lambda, \theta, \phi}}}fdl, (\lambda, \theta, \phi)\in\Bbb R^{+}\times\left[-\pi / 2, \pi / 2\right)\times[-\pi,\pi),\end{equation}
where $\Bbb R^{+}$ denotes the ray $[0,\infty)$ and $dl$ is the infinitesimal length measure on the line $l$, then our aim is to express the function $f$ via $G$.

For a function $F$, defined in $\Bbb R^{+}$, define the Mellin transform $\mathcal{M}F$ of $F$ by
\vskip-0.2cm
\begin{equation}\hskip-7cm(\mathcal{M}F)(s) = \int_{0}^{\infty}y^{s - 1}F(y)dy, \Re s > 0\end{equation}
where it should be noted that the above integral might not converge for every complex number $s$ satisfying $\Re s > 0$. For the Mellin transform we have the following inversion formula (see \cite{3}, Chap. 8.2):
\vskip-0.2cm
\begin{equation}\hskip-4cm F(r) = \mathcal{M}^{-1}(\mathcal{M}F)(r) = \frac{1}{2\pi i}\int_{\varrho - i\infty}^{\varrho + i\infty}r^{-s}(\mathcal{M}F)(s)ds,\end{equation}
where $\varrho > 0$. The inversion formula (2.24) is valid for many types of functions whereas in our case we will justify its application in the particular case where $F$ is continuous and has compact support (see Lemma 2.7). We will also need the following scaling property
\vskip-0.2cm
\begin{equation}\hskip-2.1cm \textrm{if } F_{a}(x) = F(ax) \textrm{ where } a > 0 \textrm{ then } (\mathcal{M}F_{a})(s) = a^{-s}(\mathcal{M}F)(s)\end{equation}
(see \cite{3}, Chap. 8.3). Formula (2.25) is valid in any domain for which the Mellin transforms of $F$ and $F_{a}$ exist and hence we will have to justify the existence of these transforms, in the corresponding domain, each time this scaling property is used. If $F$ is a function of two variables then we denote by $\mathcal{M}_{i}F, i = 1, 2$ the Melling transform of $F$ with respect to its $i^{th}$ variable. Finally, for every integer $n$ and complex numbers
$\zeta, \xi$ define the following integral
\vskip-0.2cm
\begin{equation}\hskip-6.2cm\mathcal{I}_{n}(\zeta, \xi) = \int_{0}^{\infty}(1 + t^{2})^{-\frac{\zeta}{2}}t^{-\xi}(1 + it)^{n}dt\end{equation}
where $\mathcal{I}_{n}$ might not converge for every $\zeta$ and $\xi$. Now we can formulate Theorem 2.5.

\begin{thm}

Let $f$ be a continuous function, defined in $\Bbb R^{3}$, which is compactly supported in the half space $x_{3} > 0$ and let $G$ be defined as in (2.22). Let
\begin{equation}f(r\cos\varphi, r\sin\varphi, \tau) = \sum_{n = -\infty}^{\infty}f_{n}(r, \tau)e^{in\varphi}, G(\lambda, \theta, \phi) = \sum_{n = -\infty}^{\infty}G_{n}(\lambda, \theta)e^{in\phi}\end{equation}
be the Fourier expansions of $f$ and $G$ respectively in the variables $\varphi$ and $\phi$. Define
$$\hskip-5.65cm G_{n}^{\ast}(\lambda, s) = G_{n}(\lambda , \arctan s)/\sqrt{1 + s^{2}}, \lambda, s\geq0,$$
$$\hskip-4.05cm G_{n}^{\ast\ast}(\lambda, \xi) = \lambda^{\xi - n - 1}\mathcal{M}_{2}\left[G_{n}^{\ast}\right](\lambda, \xi), \lambda\geq0, 0 < \Re(\xi) < 1,$$
$$\hskip-3.15cm U_{n}(\zeta, \xi) = \mathcal{M}_{1}\left[G_{n}^{\ast\ast}\right](\zeta, \xi) / \mathcal{I}_{n}(\zeta, \xi), \Re(\zeta) > n', 0 < \Re(\xi) < 1$$
where $n' = \max(0, n + 1)$, then
\vskip-0.2cm
\begin{equation}\hskip-4.9cm f_{n}(r, \tau) = r^{n}\mathcal{M}_{2}^{-1}\left[\mathcal{M}_{1}^{-1}\left[U_{n}(\cdot, \xi)\right]\right](r, \tau), r, \tau\geq0.\end{equation}
\end{thm}

\vskip0.3cm

\begin{remark} When using the Mellin inversion formula (2.24) in equation (2.28) then when using $\mathcal{M}_{1}^{-1}$, for which integration is taken with respect to the variable $\zeta$, the point $\varrho$ must be taken from the ray $(n', \infty)$ whereas when using $\mathcal{M}_{2}^{-1}$, for which integration is taken with respect to the variable $\xi$, the point $\varrho$ must be taken from the interval $(0, 1)$. This is to ensure that integration is taken over the domain of definition of $U_{n} = U_{n}(\zeta, \xi)$ for which the integral $\mathcal{I}_{n}$ converges.\end{remark}

\textbf{Proof of Theorem 2.5:} In the domain of definition of the function $G$ let us assume for now that the variable $\theta$ is restricted to the interval $[0,\pi / 2]$. Then, integrating $f$ on the line $l_{\lambda, \theta, \phi}$ with its parametrization given in equation (2.21) and using the expansion (2.27) of $f$ we obtain
$$G(\lambda, \theta, \phi) = \int_{_{l_{\lambda, \theta, \phi}}}fdl = \sum_{n = -\infty}^{\infty}\int_{-\infty}^{\infty}f_{n}\left(\sqrt{\lambda^{2} + t^{2}\cos^{2}\theta}, t\sin\theta \right)e^{in\varphi(\lambda, t, \theta, \phi)}dt$$
\begin{equation}\hskip-0.55cm = \sum_{n = -\infty}^{\infty}\int_{0}^{\infty}f_{n}\left(\sqrt{\lambda^{2} + t^{2}\cos^{2}\theta}, t\sin\theta \right)e^{in\varphi(\lambda, t, \theta, \phi)}dt\end{equation}
where
$$\cos\varphi(\lambda, t, \theta, \phi) = \frac{\lambda\cos\phi - t\cos\theta\sin\phi}{\sqrt{\lambda^{2} + t^{2}\cos^{2}\theta}}, \sin\varphi(\lambda, t, \theta, \phi) = \frac{\lambda\sin\phi + t\cos\theta\cos\phi}{\sqrt{\lambda^{2} + t^{2}\cos^{2}\theta}}$$
and where in the last passage of equation (2.29) we used the fact that $f$ is supported on the half-space $x_{3} > 0$ and that $\sin\theta\geq0$ if $0\leq\theta\leq\pi / 2$. The following relation
$$\hskip-9.2cm G(\lambda, -\theta, \phi) = \overline{G(\lambda, \theta, -\phi)}$$
can be easily checked and thus, assuming that $\lambda\geq0$ and $-\pi\leq \phi < \pi$, negative values of $\theta$ do not give any new information on $f$ and hence we will assume from now on that $\theta$ is given only in the interval $[0,\pi / 2]$. Now, define the variable $\alpha$ so that
$$\hskip-2.75cm\cos\alpha(\lambda, t, \theta) = \frac{\lambda}{\sqrt{\lambda^{2} + t^{2}\cos^{2}\theta}}, \sin\alpha(\lambda, t, \theta) = \frac{t\cos\theta}{\sqrt{\lambda^{2} + t^{2}\cos^{2}\theta}}$$
then we have
\vskip-0.2cm
$$\hskip-5.6cm\cos\varphi = \cos\alpha\cos\phi - \sin\alpha\sin\phi = \cos(\phi + \alpha),$$ $$\hskip-5.6cm\sin\varphi = \cos\alpha\sin\phi + \sin\alpha\cos\phi = \sin(\phi + \alpha).$$
Hence,
\vskip-0.2cm
$$\hskip-9cm e^{in\varphi(\lambda, t, \theta, \phi)} = e^{in\alpha(\lambda, t, \theta)}e^{in\phi}$$
and thus we have
$$\hskip-2cm G(\lambda, \theta, \phi) = \sum_{n = -\infty}^{\infty}e^{in\phi}\int_{0}^{\infty}f_{n}\left(\sqrt{\lambda^{2} + t^{2}\cos^{2}\theta}, t\sin\theta \right)e^{in\alpha(\lambda, t, \theta)}dt$$
from which we obtain that
\vskip-0.2cm
$$\hskip-7cm G_{n}(\lambda, \theta) = \frac{1}{2\pi}\int_{-\pi}^{\pi}G(\lambda, \theta, \phi)e^{-in\phi}d\phi$$ $$\hskip-5.4cm = \int_{-\infty}^{\infty}f_{n}\left(\sqrt{\lambda^{2} + t^{2}\cos^{2}\theta}, t\sin\theta \right)e^{in\alpha(\lambda, t, \theta)}dt$$
\begin{equation}\hskip-4.55cm = \int_{0}^{\infty}g_{n}\left(\sqrt{\lambda^{2} + t^{2}\cos^{2}\theta}, t\sin\theta \right)(\lambda + it\cos\theta)^{n}dt\end{equation}
where we denote
\vskip-0.2cm
$$\hskip-9.5cm g_{n}(r, \tau) = r^{-n}f_{n}(r, \tau).$$
Making the change of variables $t' = t\cos\theta$ in the right hand side of equation (2.30) we obtain that
\vskip-0.2cm
$$\hskip-3.7cm G_{n}(\lambda, \theta) = \int_{0}^{\infty}g_{n}\left(\sqrt{\lambda^{2} + (t')^{2}}, t'\tan\theta \right)(\lambda + it')^{n}\frac{dt'}{\cos\theta}.$$
If we denote $s = \tan\theta$ then since $0\leq \theta \leq \pi / 2$ it follows that $s\geq0$ and, since $\cos\theta = 1 / \sqrt{1 + \tan^{2}\theta}$ in this domain of $\theta$, we can write
\vskip-0.2cm
$$\hskip-9.35cm G_{n}(\lambda, \arctan s) / \sqrt{1 + s^{2}}$$
$$\hskip-4.5cm = \int_{0}^{\infty}g_{n}\left(\sqrt{\lambda^{2} + (t')^{2}}, st' \right)(\lambda + it')^{n}dt', s, \lambda\geq0.$$
Hence, if we define $G_{n}^{\ast}$ by the following relation $G_{n}^{\ast}(\lambda, s) = G_{n}(\lambda, \arctan s) / \sqrt{1 + s^{2}}$ and then take the Mellin transform with respect to the variable $s$ on both sides of the last equation we obtain that
\begin{equation}\hskip-1.7cm\mathcal{M}_{2}\left[G_{n}^{\ast}\right](\lambda, \xi) = \int_{0}^{\infty}\mathcal{M}_{2}\left[g_{n}\right]\left(\sqrt{\lambda^{2} + (t')^{2}}, \xi\right)(t')^{-\xi}(\lambda + it')^{n}dt'\end{equation}
where $\lambda\geq0, 0 < \Re(\xi) < 1$. The application of the Mellin scaling property (2.25) is justified here since for every $\lambda, t'\geq0$ the function
\vskip-0.2cm
$$\hskip-8.7cm s\mapsto g_{n}\left(\sqrt{\lambda^{2} + (t')^{2}}, st'\right)$$
is compactly supported in $s$ (since $f$ is compactly supported in $\Bbb R^{3}$ and thus obviously $g_{n}$ is compactly supported with respect to both variables $r$ and $\tau$) and thus, using Lemma 2.7 (see the end of this section), it follows that its Mellin transform exists for every $\Re\xi > 0$. The restriction that $0 < \Re(\xi) < 1$ is to ensure that the integral $\mathcal{I}_{n}$, as defined by equation (2.26), which will appear later in the computation, converges.

Making the change of variables $t' = \lambda t$ in the right hand side of equation (2.31) yields
\begin{equation}\hskip-1.7cm\mathcal{M}_{2}\left[G_{n}^{\ast}\right](\lambda, \xi) = \lambda^{n + 1}\lambda^{-\xi}\int_{0}^{\infty}\mathcal{M}_{2}\left[g_{n}\right]\left(\lambda\sqrt{1 + t^{2}}, \xi\right)t^{-\xi}(1 + it)^{n}dt.\end{equation}
Denoting
\vskip-0.2cm
$$\hskip-8cm G_{n}^{\ast\ast}(\lambda, \xi) =  \lambda^{\xi - n - 1}\mathcal{M}_{2}\left[G_{n}^{\ast}\right](\lambda, \xi)$$
and then taking the Mellin transform, with respect to the variable $\lambda$, on the integral in the right hand side of equation (2.32) and on $G_{n}^{\ast\ast}$ yields
$$\hskip-2cm\mathcal{M}_{1}\left[G_{n}^{\ast\ast}\right](\zeta, \xi) = \int_{0}^{\infty}\mathcal{M}_{1}\left[\mathcal{M}_{2}\left[g_{n}\right]\right]\left(\zeta, \xi\right)(1 + t^{2})^{-\frac{\zeta}{2}}t^{-\xi}(1 + it)^{n}dt$$
$$\hskip-1.8cm = \mathcal{M}_{1}\left[\mathcal{M}_{2}\left[g_{n}\right]\right]\left(\zeta, \xi\right)\underset{\mathcal{I}}{\underbrace{\int_{0}^{\infty}(1 + t^{2})^{-\frac{\zeta}{2}}t^{-\xi}(1 + it)^{n}dt}}$$
where the application of the Mellin scaling property, for the variable $\lambda$, can be justified in the exact same way as we previously did for the variable $s$. Observe that since $\Re(\zeta) > n'$ and  $0 < \Re(\xi) < 1$ then the integral $\mathcal{I} = \mathcal{I}_{n}(\zeta, \xi)$ converges. Indeed, at $t\rightarrow0^{+}$ the integrand behaves like $t^{-\xi}$ and for $t\rightarrow\infty$ if $n\leq -1$ then $\mathcal{I}_{n}$ obviously converges while for $n\geq0$ the integrand behaves like $t^{-\zeta - \xi + n}\leq t^{-\xi - 1}$ and $\mathcal{I}_{n}$ converges also for this case. Dividing both sides of the last equation by $\mathcal{I}_{n}(\zeta, \xi)$ and then taking the inverse Mellin transform in both variables $\zeta$ and $\xi$ we can obtain Theorem 2.5.

The application of the Mellin inversion formula (2.24) for $g_{n}$ and for $\mathcal{M}_{2}\left[g_{n}\right]$ can be justified since these functions have compact support in the corresponding variables on which the Mellin transforms are applied ($g_{n}$ is finitely supported in its second variable and $\mathcal{M}_{2}\left[g_{n}\right]$ is finitely supported in its first variable). Hence, the justification of these inversion formulas are now an easy consequence of Lemma 2.7.

$\hskip14.4cm\square$

\begin{lem}

Let $F:\Bbb R^{+}\rightarrow\Bbb R$ be a continuous function with compact support. Then, the Mellin transform of $F$ exists on each line $l_{c}:c + it, t\in\Bbb R$ where $c > 0$ and one can use the Mellin inversion formula (2.24) for $\varrho = c$ in order to reconstruct $F$.

\end{lem}

\begin{proof}

Let us evaluate the Mellin transform of $F$ on the line $l_{c}$. Assume that $F$ is supported inside the interval $[0,a]$ where $a > 0$, then
$$\hskip-4cm(\mathcal{M}F)(c + it) = \int_{0}^{a}y^{c + it - 1}F(y)dy = [y = e^{- x}, dy = e^{- x}dx]$$
$$\hskip1.5cm = \int_{-\ln a}^{\infty}e^{-(it + c)x}e^{x}F\left(e^{-x}\right)e^{-x}dx = \int_{-\ln a}^{\infty}e^{-itx}e^{-cx}F\left(e^{-x}\right)dx$$
$$\hskip-5.7cm = \int_{-\infty}^{\infty}e^{-itx}F^{\ast}(x)dx$$
where
\vskip-0.2cm
\[\hskip-3.5cm F^{\ast}(x) =
\begin{cases}
e^{-cx}F(e^{-x}), x\geq - \ln a,\\
0, \hskip1.75cm x < - \ln a.
\end{cases}\]
Since $c > 0$ it follows that $F^{\ast}$ is in $L_{1}(\Bbb R)$ and thus its Fourier transform exists in $\Bbb R$ which is equivalent to the fact that the Mellin transform of $F$ exists on the line $l_{c}$. Inverting $F$ from its Mellin transform (2.23) is equivalent to inverting $F^{\ast}$ from its Fourier transform and since $F$ is continuous it follows that $F^{\ast}$ is also continuous and since it is also in $L_{1}(\Bbb R)$ then it is well known that one can use the Fourier inversion formula in order to recover $F^{\ast}$. This finishes the proof of Lemma 2.7.

\end{proof}


\begin{thebibliography}{}

\bibitem{1} M. Agranovsky, V. V. Volchkov and L. A. Zalcman. \textit{Conical uniqueness sets for the spherical Radon transform}, Bulletin of the London Mathematical Society, 31, 231-236, 1999.


\bibitem{2} C. M. Cormack. \textit{Representation of a function by its line integrals, with some radiological applications}, J. Appl. Phys., 34, 2722-2727.

\bibitem{3} B. Dambaru and L. Debnath. \textit{Integral Transforms and Their Applications}, CRC Press, New York, 2007.

\bibitem{4} S. R. Deans. \textit{The Radon transform and some of its applications}, Dover Publ. Inc., Mineola, New York, 2007.

\bibitem{5} I. M. Gel'fand, S. G. Gindikin, and M. I. Graev. \textit{Selected topics in integral
geometry}, Translations of Mathematical Monographs, AMS, Providence, RI, 2003.

\bibitem{7} S. Helgason, \textit{Integral geometry and Radon transform}, Springer, New York-Dordrecht-Heidelberg-London, 2011.

\bibitem{6} S. Helgason. \textit{The Radon Transform}, Birkh¨auser, Basel 1980.

\bibitem{8} P. Mader. \textit{Uber die Darstellung von Punktfunktionen im n-dimensionalen ¨
euklidischen Raum durch Ebenenintegrale}, Math. Zeit. 26, 646-652, 1927.

\bibitem{9} F. Natterer. \textit{The Mathematics of Computerized Tomography}, Wiley, New
York, 1986.

\bibitem{10} V. P. Palamodov. \textit{Reconstruction from Integral Data}. Monographs and Research Notes in
Mathematics. CRC Press, Boca Raton, 2016.

\bibitem{11} V. P. Palamodov. \textit{Reconstructive integral geometry}. Monographs in Mathematics, 98. Birkh\"{a}user Verlag, Basel (2004)

\bibitem{12} J. Radon. \textit{\"{U}ber die Bestimmung von Funktionen durch ihre Integralwerte
l\"{a}ngs gewisserMannigfaltigkeiten}, Ber. Verh. S¨achs. Akad. Wiss.
Leipzig. Math. Nat. Kl. 69, 262-277, 1917.

\bibitem{13} B, Rubin. \textit{On the Funk-Radon-Helgason inversion method in integral geometry}, Cont.Math., 599, 175-198, 2013.

\bibitem{14} B. Rubin. \textit{Reconstruction of functions from their integrals over $k$ dimensional planes}, Israel J. of Math. 141, 93-117, 2004.

\bibitem{15} B. Rubin and Y. Wang. \textit{New inversion formulas for Radon transforms on affine Grassmannians}, arXiv:1610.02109.

\bibitem{16} V. V. Volchkov. \textit{Integral Geometry and Convolution Equations}, Kluwer Academic, Dordrecht, 2003.


\end{thebibliography}
\end{document}